\documentclass[review]{elsarticle}
\usepackage{amsmath,amsthm,amssymb,amstext}
\usepackage{lineno}
\modulolinenumbers[205]
\journal{XXXX}
\usepackage{comment}
\usepackage{amsthm}
\usepackage{setspace}
\usepackage{epstopdf}
 
\usepackage{geometry}
\usepackage{color}
\usepackage[normalem]{ulem}
\newgeometry{vmargin={21mm}, hmargin={20mm,20mm}}
\newtheorem{theorem}{Theorem}[section]
\newtheorem{lemma}[theorem]{Lemma}
\newtheorem{corollary}[theorem]{Corollary}
\newtheorem{definition}[theorem]{Definition}

\newtheorem{remark}[theorem]{Remark}
 
\begin{document}	
	\begin{frontmatter}
		\title{The windowed quadratic phase Fourier transform: structure, convolution theorem and application}
		
		\tnotetext[]{Corresponding author
			\\Email address: sargavarghese@gmail.com, manabiitism17@gmail.com}
		
		\author{Sarga Varghese, Manab Kundu*}
		
		\address{Department of Mathematics, SRM University-AP, Andhra Pradesh 522240, India}
		
		\begin{abstract}
			The windowed quadratic phase Fourier transform (WQPFT) combines the localization capabilities of windowed transforms with the phase modulation structure of the quadratic phase Fourier transform (QPFT). This paper investigates fundamental properties of the WQPFT, including linearity, shifting, modulation, conjugation, and symmetry. In addition, we derive the reproducing kernel, establish a reconstruction formula, and characterize the range of the transform. Convolution theorems in both the spectral and spatial domains are developed, along with existence results and norm estimates for the convolution operation associated with the WQPFT. Finally, as an application, the solution of a convolution equation is given using the convolution theorem of the WQPFT.
		\end{abstract}
	\end{frontmatter}

	\section{Introduction}
	The convolution theory plays a central role in diverse fields of signal and image processing particularly in the design and implementation of multiplicative filters. The classical convolution theorem establishes a fundamental duality between operations in the time and frequency domains. The convolution operation in the time domain correspond directly to multiplication operations in the frequency domain when transformed using the Fourier transform. The mathematical elegance of this relationship has made it significant in various fields of signal and image processing theory, providing computational advantages and theoretical insights that have shaped decades of technological development. The practical implications of convolution extend far beyond theoretical mathematics, enabling efficient algorithm design and implementation strategies. Digital signal processors use this relationship to perform complex filtering operations through frequency-domain multiplication, significantly reducing computational complexity compared to direct time-domain convolution \cite{con1, con2,con3, con4, con5, con6, con7, con8, con9}.
	
	The quadratic phase Fourier transform (QPFT) is a five parameter integral transform introduced by Castro et. al. \cite{lpc,lmm} in 2018, motivated by the work of Saitoh \cite{ss}, who used the reproducing kernel theory of QPFT for solving heat equation. Castro used a general quadratic function in the exponent and studied various properties of the transform. QPFT is a generalization of several well known transforms like Fourier transform (FT), fractional Fourier transform (FrFT) and linear canonical transform (LCT). QPFT has gained its applications in several disciplines of science and engineering, including harmonic analysis, quantum mechanics, differential equations, optics, pattern recognition, and so on. Lately, Prasad and Sharma \cite{Prasad1, Prasad2} expanded on the QPFT theory, developed the wavelet transform, and studied its underlying features.  Also, there have been many significant works going on in QPFT theory like \cite{Kumar, tml, dis, wp, bah, var}.
	
	The QPFT is effective for stationary signal analysis but is limited in handling non-stationary signals, as its global kernel cannot capture localized frequency variations. To overcome this, the windowed QPFT introduces a time-localized window function, allowing analysis of time-varying signals. By applying a fixed length window at successive intervals, it isolates portions of the signal that are approximately stationary, making local time-frequency representation possible. This makes it well-suited for a wide range of real world applications, where signals are often inherently time-varying. The information of the instantaneous frequency during a time interval can be obtained with the use of a suitable window, the related windowed QPFT (WQPFT) is ideal to get the frequency contents of a function.
	
	The literature on time-frequency analysis reveals a progressive development of integral transforms aimed at improving the simultaneous resolution of time and frequency content in signals \cite{tfa}. The windowed FT (WFT), one of the earliest such tools, extends the classical FT by incorporating a window function, enabling localized analysis of non-stationary signals \cite{wft0, wft1, wft2, wft3}. Its applications include fringe pattern analysis in optical metrology and the study of almost periodic signals \cite{wft4, wft5, wft6}. The WFT was later generalized to more flexible frameworks, including the windowed LCT (WLCT), which inherits properties such as orthogonality, invertibility, and covariance while offering enhanced localization and resolution \cite{wlct1, wlct2, wlct3, gao, wlct4, wlct5, wlct6, wlct7, wlct8}. Gao \cite{gao} established spectral and spatial convolution theorems for the WLCT, along with existence conditions and applications to convolution equations. Further extensions led to the windowed FrFT (WFrFT), enabling fine control over time-frequency resolution and facilitating developments in FIR filter design, graph signal processing, and aliasing error estimation \cite{wfrft1, wfrft3, wfrft5,wfrft6}. Recent advances include the fractional Gabor transform (FrGT) \cite{wfrft7} and novel windowed transforms such as the WLCT in quaternion and biquaternion settings \cite{wlct6}, which establish parity, shift, and uncertainty principles in higher-dimensional contexts. Parallel to these, QPFT have emerged for signals with nonlinear phase characteristics, leading to further developments such as the quadratic phase wave packet transform \cite{wqpft1}, the quaternion windowed QPFT \cite{wqpft2}, and the windowed octonion QPFT  \cite{wqpft3}. These generalizations incorporate a quadratic phase modulation and have been studied for their mathematical structure, convolution properties, and time-frequency localization capabilities. A comprehensive survey by Gupta and Verma \cite{wqpft4} summarizes these advancements, highlighting the need for transforms like the WQPFT, which unify windowing and quadratic-phase modulation within a common framework suitable for modern signal analysis.
	
	Owing to the significance of WQPFT and motivated by the works \cite{gao, wlct1, wlct2, Prasad3} in WLCT, this paper aims to investigate the following aspects of the transform:
	\begin{itemize}
		\item To study the fundamental properties of the WQPFT, including basic operational properties such as linearity, shifting, modulation, conjugation, and symmetry, as well as structural results like the reproducing kernel, reconstruction formula, and characterization of its range.
		
		\item To establish convolution theorems for the WQPFT in both the spectral and spatial domains.
		
		\item To prove existence theorems and derive estimates related to the convolution operation within the WQPFT framework.
		
		\item To demonstrate applications of the WQPFT in signal analysis. 
	\end{itemize}
	
	The paper is organized as follows: Section 2 provides basic concepts and notations of FT, QPFT, and WQPFT. Section 3 discusses various properties of the WQPFT, including its reconstruction formula, characterization range, and reproducing kernel. Section 4 focuses on convolution theorems - both spectral and spatial - associated with the WQPFT. Section 5 establishes existence theorems and provides estimates related to the convolution of the WQPFT. Section 6 explores applications of the WQPFT. Finally, concluding remarks are provided in Section 7.
	
	
	\section{Preliminaries} 
	In this section, we briefly review the fundamental concepts and notations, which will be used throughout the paper. We begin by recalling the classical FT, which serves as the foundation for various generalized transforms considered in this work.
	\begin{definition}
		The FT of a function $f \in L^1(\mathbb{R})$ is defined by 
		\begin{eqnarray}
			({F}f)(\omega) = \frac{1}{\sqrt{2\pi}} \int_{\mathbb{R}} e^{-ix\omega}f(x)dx,\hspace{3mm} \forall\omega\in \mathbb{R}.
		\end{eqnarray}
		Reconstruction formula is given by
		\begin{eqnarray}
			f(x)=\frac{1}{\sqrt{2\pi}} \int_{\mathbb{R}} e^{ix\omega}({F}f)(\omega)d\omega,\hspace{3mm} \forall x\in \mathbb{R}.
		\end{eqnarray}
	\end{definition}
	
	The QPFT generalizes the FT by introducing a quadratic phase factor.
	\begin{definition}\label{def1}
		Let $f \in L^1(\mathbb{R}) \cap L^2(\mathbb{R})$. Then for a given set of parameters $\Lambda=\{ a,b,c,d,e\}$, $b\neq 0$, the quadratic phase Fourier transform (QPFT) of $f$  can be defined as \cite{Prasad1} \cite{var}
		\begin{eqnarray}
			(\mathcal{Q}_\Lambda f)(\omega) = (\mathcal{Q}^{a, b, c}_{d, e} f)(\omega) =  \int_{\mathbb{R}} \mathcal{K}^{a, b, c}_{d, e}(\omega,x) f(x) dx, \hspace{3mm} \omega\in \mathbb{R},
		\end{eqnarray}
		where
		\begin{eqnarray}\label{k}
			\mathcal{K}^{a, b,  c}_{d, e} (\omega, x)= \mathcal{K}_\Lambda (\omega, x) = \sqrt{\frac{b}{2\pi i}}~e^{i(a x^2 + b x \omega + c \omega^2 + d x + e\omega)}
		\end{eqnarray}
		is the kernel of QPFT and	$ a, b, c, d, e \in \mathbb{R}$.
		Then the inversion of QPFT is given by
		\begin{eqnarray}
			f(x)=  \Big(\mathcal{Q}^{-c, -b, -a}_{-e, -d} (\mathcal{Q}^{a, b, c}_{d, e} f)\Big)(x)=  \int_{\mathbb{R}}  \mathcal{K}^{-c, -b,  -a}_{-e, -d} (x, \omega) (\mathcal{Q}^{a, b, c}_{d, e} f)(\omega) d\omega,
		\end{eqnarray}
		where
		\begin{eqnarray*}
			\mathcal{K}^{-c, -b,  -a}_{-e, -d} (x, \omega)= \mathcal{K}_{-\Lambda} (\omega, x) = \sqrt{\frac{bi}{2\pi}}~e^{-i(a x^2 + b x \omega + c \omega^2 + d x + e\omega)}.
		\end{eqnarray*}
	\end{definition}
	In analogy with the classical FT, the following result describes the asymptotic behavior of the QPFT of functions.
	\begin{lemma}(Riemann-Lebesgue)\label{lem1} 
		Let $f$ be a function in $L^1(\mathbb{R})$. Then, the QPFT of $f$, $\mathcal{Q}_\Lambda f \in \mathcal{C}_0(\mathbb{R})$.
		\\i.e., 
		\begin{eqnarray}
			(\mathcal{Q}_\Lambda f)(\omega) \to 0 \;\;\;\; as \;\;\;\; |\omega| \to \infty.
		\end{eqnarray}		
	\end{lemma}
	
	The following theorem shows that the QPFT preserves the inner product of functions in $L^2(\mathbb{R})$.
	\begin{theorem}\label{nm}
		The inner product of QPFT of two functions $f,g \in L^2 (\mathbb{R})$ is equal to the inner product of the corresponding original functions. 
		\\i.e.,
		\begin{eqnarray}\label{parseval}
			\big< f, g \big> &=& \big< \mathcal{Q}_\Lambda f, \mathcal{Q}_\Lambda g\big>.
		\end{eqnarray}	
		In particular, if $f = g$, the inner product identity reduces to a norm-preserving statement. i.e.,
		\begin{eqnarray}
			||f||_2^2 &=& ||\mathcal{Q}_\Lambda f||_2^2. 
		\end{eqnarray}
	\end{theorem}
	
	To overcome the lack of time localization in the FT, the WFT introduces a window function to analyze local frequency content of signals.
	\begin{definition}\label{wft}
		Let $\phi \in L^2(\mathbb{R}) \setminus \{0\}$ be a window function. The windowed Fourier transform (WFT) of a signal $f \in L^2(\mathbb{R})$ with respect to $\phi$ is defined by \cite{tfa}
		\begin{eqnarray*}
			\mathcal{F}_\phi [f(t)] (u,w) &=& \int_{\mathbb{R}} f(x) \overline{T_u(\phi)(x)} e^{i\omega x}dx = \int_{\mathbb{R}} f(x) \overline{\phi (x- u)} 	e^{i\omega x} dx,
		\end{eqnarray*}
		where $T_u$ is the translation operator defined by $T_u(\phi)(x)= \phi(x- u)$. The inverse of the WFT is defined by \cite{Prasad3}
		\begin{eqnarray}\label{inwft}
			f(x) = \frac{1}{2 \pi \langle \phi, \psi \rangle} \int_{\mathbb{R}} \int_{\mathbb{R}} \mathcal{F}_\phi[f] ( u, \omega) e^{-i\omega x} \psi(x -u) d\omega du, \text{ where } \psi \in L^2(\mathbb{R}) \setminus \{0\}.
		\end{eqnarray}
		If $\phi=\psi$, this becomes
		\begin{eqnarray*}
			f(x) = \frac{1}{2 \pi ||\phi||^2} \int_{\mathbb{R}} \int_{\mathbb{R}} \mathcal{F}_\phi[f] ( u, \omega) e^{-i\omega x} \phi(x -u) d\omega du.
		\end{eqnarray*}
	\end{definition}

	Combining the ideas of the QPFT and WFT, the WQPFT provides a framework that offers both quadratic phase modulation and time-frequency localization. We now introduce its definition and notations.
	\begin{definition}\label{wqpft}
		Let $\phi \in L^2(\mathbb{R}) \setminus \{0\}$ be a window function. Then the WQPFT of a signal $f \in L^2(\mathbb{R})$ with respect to $\phi$ is defined by
		\begin{eqnarray*}
			\mathcal{Q}_\phi^\Lambda [f(x)] (u,w) &=& \int_{\mathbb{R}} f(x) \overline{T_u(\phi)(x)} 	\mathcal{K}_\Lambda (\omega, x)dx
			\\&=& \int_{\mathbb{R}} f(x) \overline{\phi (x- u)} 	\mathcal{K}_{\Lambda} (\omega, x)dx,
		\end{eqnarray*}
		where $\mathcal{K}_\Lambda (\omega, x)$ is given by \eqref{k} and $T_u$ as in the Definition \ref{wft}.
	\end{definition}
	\begin{remark}\label{rem} 
		The WQPFT of a function $f$ with respect to $\phi$ can be written as the QPFT of the function $f T_u (\overline{\phi})$. 
		\begin{eqnarray*}
			\mathcal{Q}_\phi^\Lambda [f(x)] (u,\omega) &=& \mathcal{Q}_\Lambda [f\overline{T_u(\phi)}(x)] (\omega). 
		\end{eqnarray*}
	\end{remark}
	
	\begin{remark} \cite{wlct1}
		If $f \in L^2(\mathbb{R})$ and $\phi \in L^2(\mathbb{R}) \setminus \{0\}$, the WQPFT of $f$ with respect to window function $\phi$ is well defined. Also, if $f \in L^p(\mathbb{R})$ and $\phi \in L^q(\mathbb{R}) \setminus \{0\}$, $1 \leq p, q < \infty$ and $\frac{1}{p} + \frac{1}{q} = 1$. then by Holder's inequality, $f(x)\overline{\phi(x-u)} \in L^1(\mathbb{R})$ and thus WQPFT of $f$ with respect to $\phi$ is well defined. 
	\end{remark}
	
	\begin{remark}
		When $\Lambda = (0, 1, 0, 0, 0)$, the WQPFT reduces to WFT of a function with respect to $\phi$. 
	\end{remark}
	
	As an extension to the QPFT case, we now state the Riemann–Lebesgue lemma for the WQPFT, which characterizes the decay of the transform of integrable functions.
	\begin{lemma}(Riemann-Lebesgue lemma for WQPFT) \label{rl}
		For $f \in L^1(\mathbb{R}) \cap L^2(\mathbb{R})$, the WQPFT of with respect to $\phi \in L^2(\mathbb{R})\setminus\{0\}$, $\mathcal{Q}_\phi^\Lambda [f(x)] (u,\omega) \to 0$ as $|\omega| \to \infty$.
	\end{lemma}
	\begin{proof}
		By Remark \ref{rem} and Lemma \ref{lem1} the proof follows.
	\end{proof}
	
	\section{Properties of the WQPFT}
	In this section, we establish several fundamental properties of the WQPFT, including linearity, time and frequency shifting, modulation, conjugation, uniform continuity, and many more. These results helps in understanding the behavior of the transform under elementary operations.
	\subsection{Basic properties}
	Let $f_1, f_2 \in L^2(\mathbb{R})$ and $\phi, \psi \in L^2(\mathbb{R}) \setminus \{0\}$. Then
	\begin{enumerate}
		\item Linearity:
		\begin{eqnarray*}
			\mathcal{Q}_\phi^\Lambda [(mf_1 + nf_2)(x)] (u,\omega) = m \mathcal{Q}_\phi^\Lambda [f_1(x)] (u,w) + n \mathcal{Q}_\phi^\Lambda [f_2(x)] (u,\omega),
		\end{eqnarray*}
		where $m$ and $n$ are arbitrary constants.
		\item Conjugate linearity in the window function:
		\begin{eqnarray*}
			\mathcal{Q}_{m\phi+ n\psi}^\Lambda [f(x)] (u,\omega) = \overline{m} \mathcal{Q}_\phi^\Lambda [f(x)] (u, \omega) + \overline{n} \mathcal{Q}_\psi^\Lambda [f(x)] (u, \omega),
		\end{eqnarray*}
		where $m$ and $n$ are arbitrary constants.
		\item Time shifting:
		\begin{eqnarray*}
			\mathcal{Q}_{\phi}^\Lambda [f(x-k)] (u,\omega) &=& \int_{\mathbb{R}} f(x-k) \overline{\phi (x- u)} 	\sqrt{\frac{b}{2\pi i}}~e^{i(a x^2 + b x \omega + c \omega^2 + d x + e\omega)} dx.
		\end{eqnarray*}	
		Substituting $x-k = y$, the equation becomes 
		\begin{eqnarray*}			
			\mathcal{Q}_{\phi}^\Lambda [f(x-k)] (u,\omega) 
			&=& \int_{\mathbb{R}} f(y) \overline{\phi ((y+k)- u)} \sqrt{\frac{b}{2\pi i}}~e^{i(a (y+k)^2 + b (y+k) \omega + c \omega^2 + d (y+k) + e\omega)} dy
			\\&=& \int_{\mathbb{R}} f(y) \overline{\phi (y- (u-k))} e^{i(a (2yk+k^2) + b k \omega + d k)} \mathcal{K}_{\Lambda} (\omega, y) dy
			\\&=& \int_{\mathbb{R}} \tilde{f}(y) \overline{\phi (y- (u-k))} e^{i(a k^2 + b k \omega + d k)} \mathcal{K}_{\Lambda} (\omega, y) dy
			\\&=& e^{i(a k^2 + b k \omega + d k)} \mathcal{Q}_{\phi}^\Lambda [\tilde{f}(y)] (u-k,\omega),
		\end{eqnarray*}
		where $\tilde{f}(y)= e^{2iayk} f(y)$.
		
		\item Modulation (frequency shift in the WQPFT domain):
		\begin{eqnarray*}
			\mathcal{Q}_{\phi}^\Lambda [e^{i \alpha x}f(x)] (u,\omega) &=& \int_{\mathbb{R}} e^{i \alpha x}f(x) \overline{\phi (x- u)} \sqrt{\frac{b}{2\pi i}}~e^{i(a x^2 + b x \omega + c \omega^2 + d x + e\omega)} dx
			\\&=& \int_{\mathbb{R}} f(x) \overline{\phi (x- u)} \sqrt{\frac{b}{2\pi i}}~e^{i(a x^2 + b x (\omega + \frac{\alpha}{b}) + c \omega^2 + d x  + e\omega)} dx
			\\&=& e^{- \frac{i}{b^2}(2bc\alpha\omega + c\alpha^2 + be\alpha)} \int_{\mathbb{R}} f(x) \overline{\phi (x- u)} \sqrt{\frac{b}{2\pi i}}~e^{i(a x^2 + b x (\omega + \frac{\alpha}{b}) + c (\omega + \frac{\alpha}{b})^2 + d x  + e(\omega + \frac{\alpha}{b}))} dx
			\\&=& e^{- \frac{i}{b^2}(2bc\alpha\omega + c\alpha^2 + be\alpha} \mathcal{Q}_{\phi}^\Lambda [f(x)] \big(u,\omega+ \frac{\alpha}{b} \big).
		\end{eqnarray*}
		\item Conjugation:
		\begin{eqnarray*}
			\mathcal{Q}_{\overline{\phi}}^\Lambda [\overline{f}(x)] (u,\omega) &=& \int_{\mathbb{R}} \overline{f}(x) \phi (x- u) \sqrt{\frac{b}{2\pi i}}~e^{i(a x^2 + b x \omega + c \omega^2 + d x + e\omega)} dx 
			\\ &=& \overline{\int_{\mathbb{R}} f(x) \overline{\phi (x- u)} \sqrt{\frac{b i}{2\pi }}~e^{-i(a x^2 + b x \omega + c \omega^2 + d x + e\omega)} dx}
			\\ &=& \overline{\mathcal{Q}_{\phi}^{-\Lambda} [f(x)] (u,\omega)}. 
		\end{eqnarray*}
		\item Parity:
		\begin{eqnarray*}
			\mathcal{Q}_{P\phi}^\Lambda [Pf(x)] (u,\omega) = e^{i2e\omega} \mathcal{Q}_{\phi}^\Lambda [e^{i2dx} f(-x)] (-u, -\omega),
		\end{eqnarray*}
		where $Pf(x) = f(-x)$ for every $f \in L^2(\mathbb{R})$.
		\begin{proof}
			Given any $f \in L^2(\mathbb{R})$, we have
			\begin{eqnarray*}
				\mathcal{Q}_{P\phi}^\Lambda [Pf(x)] (u,\omega) &=& \int_{\mathbb{R}} f(-x) \overline{\phi (-(x- u))} \sqrt{\frac{b}{2\pi i}}~e^{i(a x^2 + b x \omega + c \omega^2 + d x + e\omega)} dx
				\\&=& e^{i2e\omega}\int_{\mathbb{R}} e^{i2dx} f(-x) \overline{\phi (-x- (-u))} \sqrt{\frac{b}{2\pi i}}~e^{i(a (-x)^2 + b (-x) (-\omega) + c (-\omega)^2 + d (-x) + e(-\omega))} dx
				\\&=& e^{i2e\omega} \mathcal{Q}_{\phi}^\Lambda [e^{i2dx} f(-x)] (-u, -\omega).
			\end{eqnarray*}
		\end{proof}
		\item Switching $f$ with $\phi$:
		\begin{eqnarray*}
			\mathcal{Q}_{\phi}^\Lambda [f(x)] (u,\omega) = e^{i((a-\frac{4ca^2}{b^2})u^2 + (b - \frac{4ca}{b}) u \omega + (d - \frac{2ae}{b}) u)} \overline{\mathcal{Q}_{f}^{-\Lambda} [\phi(x)] \Big(-u,\big(\omega+\frac{2au}{b}\big)\Big)}.
		\end{eqnarray*}
		\begin{proof}
			By carefully manipulating the Definition \ref{wqpft}, we obtain
			\begin{eqnarray*}
				\mathcal{Q}_{\phi}^\Lambda [f(x)] (u,\omega) 
				&=& \overline{\int_{\mathbb{R}} \phi (x- u) \overline{f(x)} \sqrt{\frac{b}{2\pi i}}~e^{-i(a x^2 + b x \omega + c \omega^2 + d x + e\omega)} dx}.
			\end{eqnarray*}
			With the change of variable $y = x-u$, the equation becomes
			\begin{eqnarray*}
				\mathcal{Q}_{\phi}^\Lambda [f(x)] (u,\omega) 
				&=&\overline{\int_{\mathbb{R}} \phi (y) \overline{f(y + u)} \sqrt{\frac{b}{2\pi i}}~e^{-i(a (y + u)^2 + b (y + u) \omega + c \omega^2 + d (y + u) + e\omega)} dy}\\
				&=& e^{i(au^2 + b u \omega + d u)} \overline{\int_{\mathbb{R}} \phi (y) \overline{f(y + u)} \sqrt{\frac{b}{2\pi i}}~e^{-i(a y^2 + 2ayu + b y \omega + c \omega^2 + d y + e\omega)} dy}
				\\ &=& e^{i(au^2 + b u \omega + d u)} e^{-i(\frac{4cau\omega}{b} + c\frac{2au}{b}^2 + e \frac{2au}{b})}  \\&& \times \overline{\int_{\mathbb{R}} \phi (y) \overline{f(y + u)} \sqrt{\frac{b}{2\pi i}}~e^{-i(a y^2 + b y (\omega+\frac{2au}{b}) + c (\omega+\frac{2au}{b})^2 + d y + e(\omega+\frac{2au}{b}))} dy}
				\\ &=& e^{i(au^2 + b u \omega + d u)} e^{-i(\frac{4cau\omega}{b} + c (\frac{2au}{b})^2 + e \frac{2au}{b})} \overline{\mathcal{Q}_{f}^{-\Lambda} [\phi(x)] \Big(-u,\big(\omega+\frac{2au}{b}\big)\Big)}
				\\&=& e^{i((a-\frac{4ca^2}{b^2})u^2 + (b - \frac{4ca}{b}) u \omega + (d - \frac{2ae}{b}) u)} \overline{\mathcal{Q}_{f}^{-\Lambda} [\phi(x)] \Big(-u,\big(\omega+\frac{2au}{b}\big)\Big)}.		
			\end{eqnarray*}
		\end{proof}
		
		\item Time marginal constraint (Integrating WQPFT with respect to the time variable yields the QPFT multiplied by a constant): Let $\phi \in L^2(\mathbb{R}) \cap L^1(\mathbb{R}) \setminus \{0\}$. Then 
		\begin{eqnarray*}
			\int_{\mathbb{R}} \mathcal{Q}_{\phi}^\Lambda [f(x)] (u,\omega) du = C \mathcal{Q}_\Lambda [f(x)](\omega),
		\end{eqnarray*}
		where $C= \int_{\mathbb{R}} \overline{\phi (\xi)} d\xi$.
		\begin{proof}
			Integrating the equation of WQPFT in Definition \ref{wqpft} w.r.t $u$, we get
			\begin{eqnarray*}
				\int_{\mathbb{R}} \mathcal{Q}_\phi^\Lambda [f(x)] (u,w) du &=& \int_{\mathbb{R}} \int_{\mathbb{R}} f(x) \overline{\phi (x- u)} 	\mathcal{K}_\Lambda (\omega, x)dx du
				\\&=& \int_{\mathbb{R}} \int_{\mathbb{R}} f(x- \xi) \overline{\phi (\xi)} 	\mathcal{K}_\Lambda (\omega, x- \xi) dx ~ d\xi
				\\&=& \int_{\mathbb{R}} \overline{\phi (\xi)} d\xi \int_{\mathbb{R}} f(x- \xi) \mathcal{K}_\Lambda (\omega, x- \xi) dx
				\\&=& C \mathcal{Q}_\Lambda [f(x)](\omega).
			\end{eqnarray*}
		\end{proof}
		\item Uniform continuity: $\mathcal{Q}_\phi^\Lambda [f(x)] (u,\omega)$ is uniformly continuous on $\mathbb{R}^2$.
		\begin{proof}
			For $u, \omega \in \mathbb{R}$, $h$ and $k$ sufficiently small
			\begin{eqnarray} \label{uc}
				&&\hspace{-1cm}|\mathcal{Q}_\phi^\Lambda [f(x)] (u+k,\omega+h) -\mathcal{Q}_\phi^\Lambda [f(x)] (u,\omega) | \nonumber
				\\&=& \bigg| \int_{\mathbb{R}} f(x) \overline{\phi (x- u - k)} \mathcal{K}_\Lambda (\omega+h, x)dx - \int_{\mathbb{R}} f(x) \overline{\phi (x- u)} \mathcal{K}_\Lambda (\omega, x)dx \bigg| \nonumber
				\\&=& \sqrt{\frac{b}{2\pi}} \bigg| \int_{\mathbb{R}} f(x) \overline{\phi (x- u - k)} ~e^{i(a x^2 + b x (\omega+h) + c (\omega+h)^2 + d x + e(\omega+h))} dx \nonumber
				\\&&- \int_{\mathbb{R}} f(x) \overline{\phi (x- u)} ~e^{i(a x^2 + b x (\omega+h) + c (\omega+h)^2 + d x + e(\omega+h))} dx \nonumber\\&& + \int_{\mathbb{R}} f(x) \overline{\phi (x- u)} ~e^{i(a x^2 + b x (\omega+h) + c (\omega+h)^2 + d x + e(\omega+h))} dx - \int_{\mathbb{R}} f(x) \overline{\phi (x- u)} e^{i(a x^2 + b x \omega + c \omega^2 + d x + e\omega)} dx \bigg|\nonumber.
			\end{eqnarray}	
			On further calculation, we obtain
			\begin{eqnarray}
				&&\hspace{-1cm}|\mathcal{Q}_\phi^\Lambda [f(x)] (u+k,\omega+h) -\mathcal{Q}_\phi^\Lambda [f(x)] (u,\omega) | \nonumber
				\\&\leq& \sqrt{\frac{b}{2\pi}}  \int_{\mathbb{R}} \Big| \big[f(x) \overline{\phi (x- u - k)} - f(x) \overline{\phi (x- u)}\big] ~e^{i(a x^2 + b x (\omega+h) + c (\omega+h)^2 + d x + e(\omega+h))} \Big| dx \nonumber
				\\&& + \sqrt{\frac{b}{2\pi}} \int_{\mathbb{R}} \Big| f(x) \overline{\phi (x- u)} ~\big[e^{i(a x^2 + b x (\omega+h) + c (\omega+h)^2 + d x + e(\omega+h))} - e^{i(a x^2 + b x \omega + c \omega^2 + d x + e\omega)}\big] \Big| dx \nonumber
				\\&\leq& \sqrt{\frac{b}{2\pi}}  \int_{\mathbb{R}} \big| f(x) \big| ~\big| \overline{\phi (x- u - k)} - \overline{\phi (x- u)}\big| dx \nonumber \\&&+ \sqrt{\frac{b}{2\pi}} \int_{\mathbb{R}} \big| f(x) \big| ~ \big|\overline{\phi (x- u)} \big| ~ \big| e^{i(b xh + 2c \omega h + c h^2 + eh)} -1  \big| dx.				
			\end{eqnarray}
			For the first integral in \eqref{uc}, applying Holder's inequality, we have 
			\begin{eqnarray*}
				\int_{\mathbb{R}} \big| f(x) \big| ~\big| \overline{\phi (x- u - k)} - \overline{\phi (x- u)}\big| dx  \leq ||f||_2 ||\overline{\phi} (*-u-k)-\overline{\phi}(*-u)||_2. 
			\end{eqnarray*}
			Since $C_0^\infty$ is dense in $L^2(\mathbb{R})$, for any $\epsilon >0$ and any $\phi \in L^2(\mathbb{R}) \setminus \{0\}$, there exist $\psi \in C_0^\infty$, such that $||\overline{\phi} - \overline{\psi}||_2 < \frac{\epsilon}{3}$. Thus
			\begin{eqnarray*}
				||\overline{\phi} (*-u-k)-\overline{\phi}(*-u)||_2 &\leq & ||\overline{\phi} (*-u-k)-\overline{\psi}(*-u-k)||_2 \\&&+ ||\overline{\psi} (*-u-k)-\overline{\psi}(*-u)||_2 + ||\overline{\psi} (*-u)-\overline{\phi}(*-u)||_2
				\\& < & \frac{\epsilon}{3} + \frac{\epsilon}{3} + \frac{\epsilon}{3} = \epsilon \text{ as } k \to 0.
			\end{eqnarray*}
			Therefore, for $f \in L^2(\mathbb{R})$, the first integral in \eqref{uc} approaches to zero, as $k$ tends to zero. For the second integral in \eqref{uc}, for any $x, \omega \in \mathbb{R}$, we have
			\begin{eqnarray*}
				\big| e^{i(b xh + 2c \omega h + c h^2 + eh)} -1  \big| \to 0 \text{ as } h \to 0
			\end{eqnarray*}
			and for $f \in L^2(\mathbb{R})$ and $\phi \in L^2(\mathbb{R}) \setminus \{0\}$, we obtain
			\begin{eqnarray*}
				\big| f(x) \big|  \big|\overline{\phi (x- u)} \big|  \big| e^{i(b xh + 2c \omega h + c h^2 + eh)} -1  \big| \leq 2 \big| f(x) \big|  \big|\overline{\phi (x- u)} \big|.
			\end{eqnarray*}
			Therefore by applying the dominated convergence theorem, the second integral of \eqref{uc} tends to zero, as $h \to 0$. Hence for sufficiently small $h$ and $k$, $|\mathcal{Q}_\phi^\Lambda [f(x)] (u+k,\omega+h) -\mathcal{Q}_\phi^\Lambda [f(x)] (u,\omega) |$ tends to zero. 
		\end{proof}
	\end{enumerate}
	We next present the reproducing kernel associated with the WQPFT, which allows an explicit representation of the transform, as stated in the result below. 
	\begin{theorem}(Reproducing kernel)
		Suppose that $(u_0, \omega_0)$ is any point on the 2D-plane of $(u, \omega)$, the necessary and sufficient condition that the function $\mathcal{Q}_\phi^\Lambda [f(x)] (u,\omega)$ is the WQPFT of same function is that $\mathcal{Q}_\phi^\Lambda [f(x)] (u,\omega)$ must satisfy the following reproducing kernel formula:
		\begin{eqnarray*}
			\mathcal{Q}_\phi^\Lambda [f(x)] (u_0, \omega_0) = \frac{1}{|| \phi||^2} \int_{\mathbb{R}} \int_{\mathbb{R}} \mathcal{Q}_\phi^\Lambda [f(x)] (u,\omega) ~\mathfrak{K}_\Lambda(u_0, \omega_0, u, \omega)  d\omega du,
		\end{eqnarray*}
		where $\mathcal{Q}_\phi^\Lambda [f(x)] (u_0, \omega_0)$ is the value of function $\mathcal{Q}_\phi^\Lambda [f(x)] (u, \omega)$ at $(u_0, \omega_0)$ and $\mathfrak{K}_\Lambda(u_0, \omega_0, u, \omega)$ is called the reproducing kernel that requires to satisfy
		\begin{eqnarray*}
			\mathfrak{K}_\Lambda(u_0, \omega_0, u, \omega) = \int_{\mathbb{R}} \overline{\mathcal{K}_\Lambda(\omega, x)} \phi(x -u) \overline{\phi (x- u_0)} 	\mathcal{K}_\Lambda (\omega_0, x)dx.
		\end{eqnarray*}
		
	\end{theorem}
	\begin{proof} Taking $(u_0, \omega_0)$ in place of $(u, \omega)$ in the Definition \ref{wqpft} of WQPFT, we obtain
		\begin{eqnarray*}
			\mathcal{Q}_\phi^\Lambda [f(x)] (u_0,\omega_0) = \int_{\mathbb{R}} f(x) \overline{\phi (x- u_0)} \mathcal{K}_\Lambda (\omega_0, x)dx. 
		\end{eqnarray*}
		Using the reconstruction formula given in \eqref{inverse}, the above equation becomes
		\begin{eqnarray*}
			\mathcal{Q}_\phi^\Lambda [f(x)] (u_0,\omega_0) &=& \int_{\mathbb{R}} \frac{1}{|| \phi||^2} \int_{\mathbb{R}} \int_{\mathbb{R}} \mathcal{Q}_\phi^\Lambda [f(x)] (u,\omega) \overline{\mathcal{K}_\Lambda(\omega, x)} \phi(x -u) d\omega du \overline{\phi (x- u_0)} 	\mathcal{K}_\Lambda (\omega_0, x)dx
			\\&=&  \frac{1}{|| \phi||^2} \int_{\mathbb{R}} \int_{\mathbb{R}} \mathcal{Q}_\phi^\Lambda [f(x)] (u,\omega) d\omega du \int_{\mathbb{R}} \overline{\mathcal{K}_\Lambda(\omega, x)} \phi(x -u) \overline{\phi (x- u_0)} 	\mathcal{K}_\Lambda (\omega_0, x)dx
			\\&=&  \frac{1}{|| \phi||^2} \int_{\mathbb{R}} \int_{\mathbb{R}} \mathcal{Q}_\phi^\Lambda [f(x)] (u,\omega) ~\mathfrak{K}_\Lambda(u_0, \omega_0, u, \omega)  d\omega du.
		\end{eqnarray*}
		Hence the result.
	\end{proof}
	
	\subsection{Relationship between WQPFT and WFT}
	Here we derive the relationship between the WQPFT and WFT.
	\begin{eqnarray*}
		\mathcal{Q}_\phi^\Lambda [f(x)] (u,\omega)
		&=& \int_{\mathbb{R}} f(x) \overline{\phi (x- u)} \sqrt{\frac{b}{2\pi i}}~e^{i(a x^2 + b x \omega + c \omega^2 + d x + e\omega)} dx
		\\&=& ~e^{i(c \omega^2 + e\omega)} \int_{\mathbb{R}} e^{i(b x \omega)} \sqrt{\frac{b}{2\pi i}}~e^{i(a x^2 + d x )} f(x) \overline{\phi (x- u)}  dx
		\\&=& ~e^{i(c \omega^2 + e\omega)} \int_{\mathbb{R}} e^{i(b x \omega)} \tilde{f}(x) \overline{\phi (x- u)}  dx
		\\&=& ~e^{i(c \omega^2 + e\omega)} \mathcal{F}_\phi [\tilde{f}(t)] (u, b\omega),
	\end{eqnarray*}
	where $\tilde{f} = \sqrt{\frac{b}{2\pi i}}~ e^{i(a x^2 + d x )} f(x) $ and $\mathcal{F}_\phi$ is the WFT.
	\\\\ The following result provides the inversion formula for the WQPFT, which enables the recovery of the original function from its transform.
	\begin{theorem}{(Reconstruction formula)}
		Let $f \in L^2(\mathbb{R})$ and $\phi, \psi \in L^2(\mathbb{R}) \setminus \{0\}$, the reconstruction formula for WQPFT is given by
		\begin{eqnarray*}
			f(x) = \frac{1}{\langle \phi, \psi \rangle} \int_{\mathbb{R}} \int_{\mathbb{R}} \mathcal{Q}_\phi^\Lambda [f(x)] (u,\omega) \overline{\mathcal{K}_\Lambda(\omega, x)} \psi(x -u) d\omega du.
		\end{eqnarray*} 
		If $\phi = \psi$, then 
		\begin{eqnarray}\label{inverse}
			f(x) = \frac{1}{|| \phi||^2} \int_{\mathbb{R}} \int_{\mathbb{R}} \mathcal{Q}_\phi^\Lambda [f(x)] (u,\omega) \overline{\mathcal{K}_\Lambda(\omega, x)} \phi(x -u) d\omega du.
		\end{eqnarray}
	\end{theorem}
	\begin{proof}
		Let $h(x) = \sqrt{\frac{b}{2\pi i}}~ e^{i(a x^2 + d x )} f(x) $. Then $h \in L^2(\mathbb{R})$. So by inversion formula \eqref{inwft} for WFT, we get for $\phi, \psi \in L^2(\mathbb{R}) \setminus \{0\}$,
		\begin{eqnarray*}
			h(x) &=& \frac{1}{2 \pi \langle \phi, \psi \rangle} \int_{\mathbb{R}} \int_{\mathbb{R}} \mathcal{F}_\phi[h] ( u, \omega) e^{-i\omega x} \psi(x -u) d\omega du
			\\&=& \frac{1}{2 \pi \langle \phi, \psi \rangle} \int_{\mathbb{R}} \int_{\mathbb{R}} \mathcal{F}_\phi[h] ( u, b \omega) e^{-ib\omega x} \psi(x -u) d(b \omega) du.
		\end{eqnarray*}
		i.e.,
		\begin{eqnarray*}
			\sqrt{\frac{b}{2\pi i}}~ e^{i(a x^2 + d x )} f(x) &=& \frac{1}{2 \pi \langle \phi, \psi \rangle} \int_{\mathbb{R}} \int_{\mathbb{R}} \mathcal{F}_\phi[h] ( u, b \omega) e^{-ib\omega x} \psi(x -u) d(b \omega) du
			\\&=& \frac{1}{2 \pi \langle \phi, \psi \rangle} \int_{\mathbb{R}} \int_{\mathbb{R}} e^{-i(c \omega^2 + e\omega)} \mathcal{Q}_\phi^\Lambda [f(t)] (u,\omega) e^{-ib\omega x} \psi(x -u) d(b \omega) du.
		\end{eqnarray*}
		Therefore
		\begin{eqnarray*}
			f(x) &=& \sqrt{\frac{2\pi i}{b}}~ e^{-i(a x^2 + d x )} \frac{b}{2 \pi \langle \phi, \psi \rangle} \int_{\mathbb{R}} \int_{\mathbb{R}} e^{-i(c \omega^2 + e\omega)} \mathcal{Q}_\phi^\Lambda [f(t)] (u,\omega) e^{-ib\omega x} \psi(x -u) d\omega du
			\\&=& \frac{1}{\langle \phi, \psi \rangle} \int_{\mathbb{R}} \int_{\mathbb{R}} \mathcal{Q}_\phi^\Lambda [f(t)] (u,\omega) \overline{\mathcal{K}_\Lambda(\omega, x)} \psi(x -u) d\omega du.
		\end{eqnarray*}
		Hence the proof.
	\end{proof}

	We now characterize the range of the WQPFT, identifying the necessary and sufficient conditions for an $L^2$- function to be the WQPFT of an admissible function.
	
	\begin{theorem}(Characterization of range of WQPFT)
		Let $\phi \in L^2(\mathbb{R}) \setminus \{0\}$ be a normalized window function, i.e, $||\phi||_2 = 1$, and let $h \in L^2(\mathbb{R})$. Then h belongs to the range of the WQPFT, i.e, $h \in \mathcal{Q}_\phi^\Lambda (L^2(\mathbb{R}))$ if and only if for all $(u', \omega') \in \mathbb{R}^2$, h satisfies
		\begin{eqnarray*}
			h(u', \omega') &=& \int_{\mathbb{R}} \int_{\mathbb{R}} h (u,\omega) \langle \phi_{u,\omega}, \phi_{u',\omega'} \rangle d\omega du,
		\end{eqnarray*}
		where $\phi_{u,\omega} = \overline{\mathcal{K}_\Lambda(\omega, u)} \phi(x -u)$.
		
	\end{theorem}
	\begin{proof}
		Suppose $h$ belongs to $\mathcal{Q}_\phi^\Lambda (L^2(\mathbb{R}))$. This guarantees the existence of a function $f \in L^2(\mathbb{R})$ such that $\mathcal{Q}_\phi^\Lambda (f) = h$. Therefore by the Definition of WQPFT \ref{wqpft} and the reconstruction formula given in \eqref{inverse}, we get 
		\begin{eqnarray*}
			h(u', \omega') = \mathcal{Q}_\phi^\Lambda (f) (u', \omega') &=& \int_{\mathbb{R}} f(x) \overline{\phi (x- u')} 	\mathcal{K}_\Lambda (\omega', x)dx \\&=& \int_{\mathbb{R}} \Big( \frac{1}{|| \phi||^2} \int_{\mathbb{R}} \int_{\mathbb{R}} \mathcal{Q}_\phi^\Lambda [f(t)] (u,\omega) \overline{\mathcal{K}_\Lambda(\omega, x)} \phi(x -u) d\omega du\Big) \overline{\phi (x- u')} \mathcal{K}_\Lambda (\omega', x) dx
			\\&=&   \frac{1}{|| \phi||^2} \int_{\mathbb{R}} \int_{\mathbb{R}} \mathcal{Q}_\phi^\Lambda [f(t)] (u,\omega) \Big( \int_{\mathbb{R}}\overline{\mathcal{K}_\Lambda(\omega, x)} \phi(x -u) \overline{\phi (x- u')} \mathcal{K}_\Lambda (\omega', x) dx \Big) d\omega du
			\\&=&   \frac{1}{|| \phi||^2} \int_{\mathbb{R}} \int_{\mathbb{R}} \mathcal{Q}_\phi^\Lambda [f(t)] (u,\omega) \Big( \int_{\mathbb{R}} \phi_{u,\omega} \overline{\phi_{u',\omega'}} dx \Big) d\omega du, 
		\end{eqnarray*}
		where $\phi_{u,\omega} = \overline{\mathcal{K}_\Lambda(\omega, u)} \phi(x -u)$. We have $||\phi||=1$. Then the above equation becomes
		\begin{eqnarray*}
			h(u', \omega') &=&  \int_{\mathbb{R}} \int_{\mathbb{R}} \mathcal{Q}_\phi^\Lambda [f(t)] (u, \omega) \langle \phi_{u,\omega}, \phi_{u',\omega'} \rangle d\omega du
			\\&=&  \int_{\mathbb{R}} \int_{\mathbb{R}} h(u, \omega)  \langle \phi_{u,\omega}, \phi_{u',\omega'} \rangle d\omega du.
		\end{eqnarray*}
		For the converse, assume $h \in L^2(\mathbb{R})$ such that
		\begin{eqnarray*}
			h(u', \omega') &=& \int_{\mathbb{R}} \int_{\mathbb{R}} h (u,\omega) \langle \phi_{u,\omega}, \phi_{u',\omega'} \rangle d\omega du.
		\end{eqnarray*}
		If
		\begin{eqnarray*}
			f(x) = \int_{\mathbb{R}} \int_{\mathbb{R}} h(u, \omega) \overline{\mathcal{K}_\Lambda (\omega, x)} \phi(x-u) du d\omega. 
		\end{eqnarray*}
		Then 
		\begin{eqnarray*}
			||f||^2 &=& \int_{\mathbb{R}} f(x \overline{f(x)}) dx
			\\&=& \int_{\mathbb{R}} \int_{\mathbb{R}} \int_{\mathbb{R}} h(u, \omega) \overline{\mathcal{K}_\Lambda (\omega, x)} \phi(x-u) du d\omega \overline{\int_{\mathbb{R}} \int_{\mathbb{R}} h(u', \omega') \overline{\mathcal{K}_\Lambda (\omega', x)} \phi(x-u') du' d\omega'} dx
			\\&=& \int_{\mathbb{R}}  \int_{\mathbb{R}}  \int_{\mathbb{R}}  \int_{\mathbb{R}}  \int_{\mathbb{R}} h(u, \omega) \phi_{u,\omega} \overline{\phi_{u',\omega'}} \overline{h(u', \omega')} du d\omega du' d\omega' dx
			\\&=& \int_{\mathbb{R}}  \int_{\mathbb{R}}  \int_{\mathbb{R}}  \int_{\mathbb{R}} h(u, \omega) \langle \phi_{u,\omega}, \phi_{u',\omega'} \rangle \overline{h(u', \omega')} du d\omega du' d\omega'
			\\&=& 	\int_{\mathbb{R}}  \int_{\mathbb{R}}  \int_{\mathbb{R}}  \int_{\mathbb{R}} h(u', \omega')  \overline{h(u', \omega')} du' d\omega'
			\\ &=& ||h||^2.
		\end{eqnarray*}
		Thus $f$ is also a square integrable function. Now for each $(u', \omega') \in \mathbb{R}^2$, by using Fubini's theorem,
		\begin{eqnarray*}
			\mathcal{Q}_\phi^\Lambda(f)(u', \omega') &=& \int_{\mathbb{R}} f(x) \overline{\phi_{u',\omega'}(x)} dx 
			\\&=& \int_{\mathbb{R}} \bigg( \int_{\mathbb{R}} \int_{\mathbb{R}} h(u, \omega) \overline{\mathcal{K}_\Lambda (\omega, x)} \phi(x-u) du d\omega \bigg) \overline{\phi_{u',\omega'}(x)} dx
			\\&=& \int_{\mathbb{R}}  \int_{\mathbb{R}} h(u, \omega) \bigg( \int_{\mathbb{R}}  \phi_{u,\omega}(x) \overline{\phi_{u',\omega'}(x)} dx  \bigg) du d\omega
			\\&=& \int_{\mathbb{R}}  \int_{\mathbb{R}} h(u, \omega) \langle \phi_{u,\omega}(x), \phi_{u',\omega'}(x) \rangle du d\omega
			\\&=& h (u', \omega').
		\end{eqnarray*}
		i.e., $\mathcal{Q}_\phi^\Lambda(f)(u', \omega') = \int_{\mathbb{R}} f(x) \overline{\phi (x- u')} 	\mathcal{K}^{a, b,  c}_{d, e} (\omega', x)dx = h$. This completes the proof.
	\end{proof}
	
	\section{Convolutions}
	In this section, we derive two convolutions- spectral and spatial- for the WQPFT. 
	\subsection{Spectral convolution theorem of the WQPFT}
	In this subsection, we will obtain the spectral convolution operator associated with WQPFT by the spectral domain representation. The conclusion is based on the fact that the WQPFT of a convolution of two functions is the product of their respective WQPFTs.
	
	\begin{definition}
		If $f \in L^p (\mathbb{R})$, $1 \leq p \leq \infty$, and $g \in L^1 (\mathbb{R})$. Then the classical convolution operator of the FT defined by 
		\begin{eqnarray}\label{clcon}
			(f * g)(x) = f(x) * g(x) = \int_{\mathbb{R}} f(\tau) g(x-\tau) d\tau
		\end{eqnarray}
		is well defined and belongs to $L^p(\mathbb{R})$. Moreover
		\begin{eqnarray}
			||f * g ||_{L^p(\mathbb{R})} \leq ||f||_p ||g||_1.
		\end{eqnarray}
	\end{definition}
	\begin{definition}
		If $f, g \in L^1 (\mathbb{R})$, then
		\begin{eqnarray}
			\hat{(f * g)}(\omega) = \sqrt{2 \pi} (\hat{ f})(\omega) (\hat{ g})(\omega),
		\end{eqnarray}
		where $\hat{ f}$ is FT of $f$.
	\end{definition}	
	
	The convolution and correlation for the QPFT, which are generalized form of convolution and correlation associated with FT is given below.
	
	\subsection{QPFT convolution operation}
	\begin{theorem}\cite{Prasad1} \label{con1}
		Let $f, g \in L^1(\mathbb{R})$. Assume $\tilde{F}(\omega) = (\mathcal{Q}_\Lambda f)(\omega)$ and $\tilde{G}(\omega)= (\mathcal{Q}_\Lambda g)(\omega)$. Then
		\begin{eqnarray}\label{cvt1}
			(\mathcal{Q}_\Lambda (f \otimes_{a} g))(\omega) = e^{-i(c \omega^2 + e\omega)} \tilde{F}(\omega) \tilde{G}(\omega),
		\end{eqnarray}
		where
		\begin{eqnarray} \label{conv}
			(f \otimes_{a} g)(x) &=& \sqrt{\frac{b}{i}} e^{-iax^2} \Big(\big(f(x) e^{i a x^2}\big) * \big(g(x)e^{iax^2}\big)\Big) \nonumber
			\\&=& \sqrt{\frac{b}{2\pi i}} e^{-ia x^2} \int_{\mathbb{R}} f(\tau)e^{ia \tau^2} g(x - \tau) e^{i(a (x- \tau)^2)}
			d\tau.
		\end{eqnarray}
	\end{theorem}
	
	\begin{remark} (WQPFT in terms of QPFT convolution) The WQPFT defined in \ref{wqpft} can be represented as a convolution-type operation as follows:
		\begin{eqnarray} \label{rel}
			\mathcal{Q}_\phi^\Lambda [f(x)] (u,\omega) &=& \int_{\mathbb{R}} f(x) \overline{\phi (x- u)} 	\mathcal{K}^{a, b,  c}_{d, e} (\omega, x)dx \nonumber
			\\&=& \sqrt{\frac{b}{2\pi i}} \int_{\mathbb{R}}  f(x) e^{i(b x \omega + c \omega^2 + e\omega)} \overline{\phi (x- u)} e^{i(a x^2 + d x)} dx \nonumber
			\\&=& \sqrt{\frac{b}{2\pi i}} \int_{\mathbb{R}}  f(x) e^{i(b x \omega + c \omega^2 + e\omega)} \overline{\phi (-(u- x))} e^{i(a (u-x)^2)} e^{i\{dx-a(u-x)^2\}} dx \nonumber
			\\&=& \sqrt{\frac{b}{2\pi i}} e^{-iau^2} \int_{\mathbb{R}}  f(x) e^{i(b x \omega + c \omega^2 + e\omega)} e^{iax^2} \overline{\phi (-(u - x))} e^{i\{dx-ax^2+2aux\}} e^{i(a (u-x)^2)}  dx \nonumber
			\\&=& \sqrt{\frac{b}{2\pi i}} e^{-iau^2} \int_{\mathbb{R}}  f_\omega(x) e^{iax^2} \phi_u (u - x) e^{i(a (u-x)^2)}  dx \nonumber
			\\&=& f_\omega \otimes_{a} \phi_u (u),
		\end{eqnarray}
		where $f_\omega(x) =  f(x) e^{i(b x \omega + c \omega^2 + e\omega)}$  and $\phi_u(u-x) =  \overline{\phi (-(u - x))} e^{i\{dx-ax^2+2aux\}}$.
	\end{remark}
	\begin{remark}
		Let $f \in L^p(\mathbb{R})$ and $g \in L^q(\mathbb{R})$ such that $\frac{1}{p} + \frac{1}{q} = 1, 1 \leq p , q \leq \infty$, then
		\begin{eqnarray}\label{norm}
			||f \otimes_{a} g ||_\infty \leq \sqrt{\frac{|b|}{2\pi}} ||f||_{L^p(\mathbb{R})} ||g||_{L^q(\mathbb{R})}.
		\end{eqnarray}
		Therefore by using equation \eqref{rel}, we get 
		\begin{eqnarray}\label{p}
			||\mathcal{Q}_\phi^\Lambda [f(x)] (u,\omega) ||_\infty \leq \sqrt{\frac{|b|}{2 \pi}}||f||_{L^p(\mathbb{R})} ||\phi||_{L^q(\mathbb{R})},
		\end{eqnarray}
		where $f \in L^p(\mathbb{R})$ and $\phi \in L^q(\mathbb{R}) \setminus \{0\}$.
	\end{remark}
	
	\begin{theorem}
		Let $\phi \in L^2(\mathbb{R}) \setminus \{0\}$ be a window function. If $\mathcal{Q}_\phi^\Lambda [f(x)] (u,\omega)$ and $\mathcal{Q}_\phi^\Lambda [g(x)] (u,\omega)$ are the WQPFTs of the functions $f, g \in L^2(\mathbb{R})$, respectively and satisfy that 
		\begin{eqnarray} \label{wconv}
			\mathcal{Q}_\phi^\Lambda [f \odot g] (u,\omega)= \mathcal{Q}_\phi^\Lambda [f] (u,\omega) \mathcal{Q}_\phi^\Lambda [g] (u,\omega).
		\end{eqnarray}
		Then the spectral convolution operator becomes
		\begin{eqnarray}
			(f \odot g)(x) &=& \int_{\mathbb{R}} \int_{\mathbb{R}} \frac{b^2 e^{-i(b x \omega + c \omega^2 + e\omega -cv^2 -ev)}}{(2\pi)^2 i \mathcal{Q}_{\Lambda} [\phi_u(x)](v)} \bigg( e^{-ia u^2} \int_{\mathbb{R}} f_\omega(\tau)e^{ia \tau^2} \phi_u(u - \tau) e^{i(a (u- \tau)^2)} d\tau \bigg) \nonumber
			\\&& \times \bigg( e^{-ia u^2} \int_{\mathbb{R}} g_\omega (\eta) e^{ia \eta^2} \phi_u (u - \eta) e^{i(a (x- \eta)^2)} d\eta \bigg) du dv.
		\end{eqnarray}
	\end{theorem}
	\begin{proof}
		Applying QPFT on equation \eqref{rel} and using the formula \eqref{cvt1}, we obtain
		\begin{eqnarray*}
			\mathcal{Q}_\Lambda [\mathcal{Q}_\phi^\Lambda [f(x)] (u,\omega)](v) &=& \mathcal{Q}_\Lambda [f_\omega \otimes_{a} \phi_u (u)](v)
			\\&=& e^{-i(c v^2 + ev)} \mathcal{Q}_\Lambda [f_\omega(x)](v) \mathcal{Q}_\Lambda [\phi_u(x)](v).
		\end{eqnarray*}
		Then
		\begin{eqnarray} \label{1}
			\mathcal{Q}_\Lambda [\mathcal{Q}_\phi^\Lambda [f \odot g (x)] (u,\omega)](v) 
			&=& e^{-i(c v^2 + ev)} \mathcal{Q}_\Lambda [(f \odot  g)_\omega(x)](v) \mathcal{Q}_\Lambda [\phi_u(x)](v) \nonumber
			\\&=& e^{-i(c v^2 + ev)} \mathcal{Q}_\Lambda [(f \odot  g)(x) e^{i(b x \omega + c \omega^2 + e\omega)}](v) \mathcal{Q}_\Lambda [\phi_u(x)](v).
		\end{eqnarray}
		Applying QPFT on both sides of the theorem statement \eqref{wconv} and by using the relation \eqref{rel}, we get
		\begin{eqnarray}\label{2}
			\mathcal{Q}_\Lambda [\mathcal{Q}_\phi^\Lambda [f \odot g] (u,\omega)](v) 
			& = & \mathcal{Q}_\Lambda [\mathcal{Q}_\phi^\Lambda [f] (u,\omega) \mathcal{Q}_\phi^\Lambda [g] (u,\omega)](v) \nonumber
			\\& = & \mathcal{Q}_\Lambda [(f_\omega \otimes_{a} \phi_u) (u) ~ (g_\omega \otimes_{a} \phi_u) (u)](v).
		\end{eqnarray}
		Equations \eqref{1} and \eqref{2} together gives
		\begin{eqnarray*} 
			e^{-i(c v^2 + ev)} \mathcal{Q}_\Lambda [(f \odot  g)(x) e^{i(b x \omega + c \omega^2 + e\omega)}](v) \mathcal{Q}_\Lambda [\phi_u(x)](v)  = \mathcal{Q}_\Lambda [(f_\omega \otimes_{a} \phi_u) (u) ~ (g_\omega \otimes_{a} \phi_u) (u)](v),
		\end{eqnarray*}
		provided $\mathcal{Q}_\Lambda [\phi_u(x)](v)  \neq 0$. So
		\begin{eqnarray*}
			\mathcal{Q}_\Lambda [(f \odot  g)(x) e^{i(b x \omega + c \omega^2 + e\omega)}](v)   =  \mathcal{Q}_{-\Lambda} [\phi_u(x)](v) ~e^{i(c v^2 + ev)}~ \mathcal{Q}_\Lambda [(f_\omega \otimes_{a} \phi_u) (u) ~ (g_\omega \otimes_{a} \phi_u) (u)](v).
		\end{eqnarray*}
		i.e.,
		\begin{eqnarray*}
			&&(f \odot  g)(x) e^{i(b x \omega + c \omega^2 + e\omega)} =  \mathcal{Q}_{-\Lambda}\big[\mathcal{Q}_{-\Lambda} [\phi_u(x)](v) ~e^{i(c v^2 + ev)}~ \mathcal{Q}_\Lambda [(f_\omega \otimes_{a} \phi_u) (u) ~ (g_\omega \otimes_{a} \phi_u) (u)](v)\big](x).
		\end{eqnarray*}
		Applying inversion formula
		\begin{eqnarray*}
			(f \odot  g)(x) e^{i(b x \omega + c \omega^2 + e\omega)}  &=& \int_{\mathbb{R}} \mathcal{K}_{-\Lambda}(u, \omega) \mathcal{Q}_{-\Lambda} [\phi_u(x)](v) ~e^{i(c v^2 + ev)}~ \mathcal{Q}_\Lambda [(f_\omega \otimes_{a} \phi_u) (u) ~~~ (g_\omega \otimes_{a} \phi_u) (u)](v) dv
			\\&=& \int_{\mathbb{R}} \mathcal{K}_{-\Lambda}(u, \omega) \mathcal{Q}_{-\Lambda} [\phi_u(x)](v) ~e^{i(c v^2 + ev)}~  \int_{\mathbb{R}} \mathcal{K}(u, \omega) (f_\omega \otimes_{a} \phi_u) (u) ~ (g_\omega \otimes_{a} \phi_u) du dv.
		\end{eqnarray*}
		By the definition of $\otimes_{a}$ in equation \eqref{conv}, we get
		\begin{eqnarray*}
			(f \odot  g)(x) e^{i(b x \omega + c \omega^2 + e\omega)}  &=& \int_{\mathbb{R}} \mathcal{K}_{-\Lambda}(u, \omega) \mathcal{Q}_{-\Lambda} [\phi_u(x)](v) ~e^{i(c v^2 + ev)}~  \int_{\mathbb{R}} \mathcal{K}(u, \omega)
			\\&& ~\times \bigg(\sqrt{\frac{b}{2\pi i}} e^{-ia u^2} \int_{\mathbb{R}} f_\omega(\tau)e^{ia \tau^2} \phi_u(u - \tau) e^{i(a (u- \tau)^2)} d\tau \bigg) 
			\\&& ~ \times \bigg(\sqrt{\frac{b}{2\pi i}} e^{-ia u^2} \int_{\mathbb{R}} g_\omega (\eta) e^{ia \eta^2} \phi_u (u - \eta) e^{i(a (x- \eta)^2)} d\eta \bigg) du dv.
		\end{eqnarray*} 
		On further calculations, this expression becomes
		\begin{eqnarray*}
			(f \odot  g)(x)  &=& \int_{\mathbb{R}} \int_{\mathbb{R}} \frac{b^2}{(2\pi)^2 i}  e^{-i(b x \omega + c \omega^2 + e\omega)} \mathcal{Q}_{-\Lambda} [\phi_u(x)](v) ~e^{i(c v^2 + ev)}~ \bigg( e^{-ia u^2} \int_{\mathbb{R}} f_\omega(\tau)e^{ia \tau^2} \phi_u(u - \tau) e^{i(a (u- \tau)^2)} d\tau \bigg) 
			\\&& \times \bigg( e^{-ia u^2} \int_{\mathbb{R}} g_\omega (\eta) e^{ia \eta^2} \phi_u (u - \eta) e^{i(a (x- \eta)^2)} d\eta \bigg) du dv
			\\&=&  \int_{\mathbb{R}} \int_{\mathbb{R}} \frac{b^2 e^{-i(b x \omega + c \omega^2 + e\omega -cv^2 -ev)}}{(2\pi)^2 i \mathcal{Q}_{\Lambda} [\phi_u(x)](v)} 
			\bigg( e^{-ia u^2} \int_{\mathbb{R}} f_\omega(\tau)e^{ia \tau^2} \phi_u(u - \tau) e^{i(a (u- \tau)^2)} d\tau \bigg) 
			\\&& \times \bigg( e^{-ia u^2} \int_{\mathbb{R}} g_\omega (\eta) e^{ia \eta^2} \phi_u (u - \eta) e^{i(a (x- \eta)^2)} d\eta \bigg) du dv.
		\end{eqnarray*}
		Hence we obtain the result.
	\end{proof}
	
	\subsection{Spatial convolution theorem of the WQPFT}
	In this section, we will obtain the spatial convolution theorem of the WQPFT by using the convolution operation of QPFT defined in \eqref{conv}.
	\begin{theorem}[Spatial convolution theorem]
		Let $\phi, \psi \in L^2(\mathbb{R}) \setminus \{0\}$ be two window functions. Then, for every functions $f, g \in L^2 (\mathbb{R})$, we have
		\begin{eqnarray*}
			\mathcal{Q}_{\phi\otimes_{a} \psi}^\Lambda (f \otimes_{2a} g) (u, \omega) = \mathcal{E}(u, \omega) \int_{\mathbb{R}}   \mathcal{Q}^\Lambda_\phi (f)(m_0, m) \mathcal{Q}^\Lambda_\psi (g)(m_1, u- m) e^{i(2aum + \frac{4a^2 c}{b^2} u m - \frac{8a^2c}{b^2} m^2 - 2am^2)} dm,
		\end{eqnarray*}	
		where
		$\mathcal{E}(u, \omega) = \bigg(\overline{\sqrt{\frac{b}{2\pi i}}}\bigg) \bigg(\sqrt{\frac{2\pi i}{b}}\bigg)^2 e^{i(-\frac{4a^2 c}{b^2} u^2 + \frac{4ac}{b} \omega u + \frac{2ae}{b} u - c \omega^2 - e \omega))}$, $m_0 = (\omega - \frac{2a}{b} u + \frac{2a}{b} m)$ and $m_1 (\omega - \frac{2a}{b} m)$ and $\otimes_{a}$ as given in \eqref{conv}.
	\end{theorem}
	\begin{proof} Invoking the equation \eqref{conv} on functions and window functions, we get
		\begin{eqnarray*}
			&&\mathcal{Q}_{\phi\otimes_{a} \psi}^\Lambda (f \otimes_{2a} g) (u, \omega) 
			\int_{\mathbb{R}}  (f \otimes_{2a} g)(x) (\overline{\phi \otimes_{a} \psi})(x- u) \mathcal{K}(\omega, x) dx
			\\&=& \int_{\mathbb{R}}  \bigg(\sqrt{\frac{b}{2\pi i}} e^{-i2a x^2} \int_{\mathbb{R}} f(\tau)e^{i2a \tau^2} g(x - \tau) e^{i(2a (x- \tau)^2)} d\tau \bigg) 
			\\&& \times \overline{\bigg(\sqrt{\frac{b}{2\pi i}} e^{-ia (x-u)^2} \int_{\mathbb{R}} \phi(r)e^{ia r^2} \psi(x - u -r) e^{i(a (x- u -r)^2)} dr \bigg)} \mathcal{K}(\omega, x) dx
			\\&=& \int_{\mathbb{R}}  \bigg(\sqrt{\frac{b}{2\pi i}} e^{-i2a x^2} \int_{\mathbb{R}} f(\tau)e^{i2a \tau^2} g(x - \tau) e^{i(2a (x- \tau)^2)} d\tau \bigg) 
			\\&& \times \bigg(\overline{\sqrt{\frac{b}{2\pi i}}} e^{ia (x- u)^2} \int_{\mathbb{R}} \overline{\phi(r)} e^{-ia r^2} \overline{\psi(x - u - r)} e^{-ia (x- u -r)^2} dr \bigg) \sqrt{\frac{b}{2\pi i}}~e^{i(a x^2 + b x \omega + c \omega^2 + d x + e\omega)} dx
			\\&=& \bigg(\sqrt{\frac{b}{2\pi i}} \bigg)^2 \bigg(\overline{\sqrt{\frac{b}{2\pi i}}} \bigg)  \int_{\mathbb{R}^3}  f(\tau) g(x - \tau)  \overline{\phi}(r)  \overline{\psi}(x - u - r) \\&& \times e^{-ia x^2}e^{i2a \tau^2}  e^{i(2a (x- \tau)^2)}  e^{ia (x- u)^2} e^{-ia r^2} e^{-ia (x- u -r)^2} 
			e^{i(b x \omega + c \omega^2 + d x + e\omega)} d\tau dr dx.
		\end{eqnarray*}
		Let	$x - \tau = x_1, r = \tau - m $. Then we obtain
		\begin{eqnarray*}
			&&\mathcal{Q}_{\phi\otimes \psi}^\Lambda (f \otimes g) (u, \omega) 
			\\&=& \bigg(\sqrt{\frac{b}{2\pi i}} \bigg)^2 \bigg(\overline{\sqrt{\frac{b}{2\pi i}}} \bigg)  \int_{\mathbb{R}^3}  f(\tau) g(x_1)  \overline{\phi}(\tau -m)  \overline{\psi}(x_1 - (u - m)) \\&& \times e^{-ia(x_1+ \tau)^2}e^{i2a \tau^2}  e^{i2a x_1^2}  e^{ia (\tau + x_1 - u)^2} e^{-ia (\tau - m)^2} e^{-ia (x_1 - (u - m))^2} 
			e^{i(b (x_1 + \tau) \omega + c \omega^2 + d (x_1 + \tau) + e\omega)} d\tau (-dm) dx_1
			\\&=& - \bigg(\sqrt{\frac{b}{2\pi i}} \bigg)^2 \bigg(\overline{\sqrt{\frac{b}{2\pi i}}} \bigg)  \int_{\mathbb{R}^3}  f(\tau) g(x_1)  \overline{\phi}(\tau -m)  \overline{\psi}(x_1 - (u - m)) \\&& \times e^{ia(\tau^2 + x_1^2 - 2 \tau u  - m^2 + 2 \tau m - m^2 - 2 x_1m + 2um)}  e^{i(b (x_1 + \tau) \omega + c \omega^2 + d (x_1 + \tau) + e\omega)} d\tau dm dx_1
			\\&=& - \bigg(\overline{\sqrt{\frac{b}{2\pi i}}}\bigg)   \int_{\mathbb{R}} \bigg(\sqrt{\frac{b}{2\pi i}} \bigg)  \int_{\mathbb{R}}  f(\tau) \overline{\phi}(\tau -m) e^{ia(\tau^2 - 2 \tau u + 2 \tau m)} 
			e^{i(b \tau \omega + d \tau )} d\tau
			\\&& \bigg(\sqrt{\frac{b}{2\pi i}} \bigg)\int_{\mathbb{R}} g(x_1)    \overline{\psi}(x_1 - (u - m)) e^{ia(x_1^2 - 2 x_1m)} 
			e^{i(b x_1\omega + d x_1)} dx_1  e^{ia(- 2m^2 + 2um)} 
			e^{i(c \omega^2 + e\omega)} dm
			\\&=& - \bigg(\overline{\sqrt{\frac{b}{2\pi i}}}\bigg)   \int_{\mathbb{R}} \bigg(\sqrt{\frac{b}{2\pi i}} \bigg)  \int_{\mathbb{R}}  f(\tau) \overline{\phi}(\tau -m) e^{i(a\tau^2  +b \tau (\omega - \frac{2a}{b} u + \frac{2a}{b} m) + d \tau )} d\tau.
		\end{eqnarray*}	
		Rearranging the terms and simplying the expression yields
		\begin{eqnarray*}
			&&\mathcal{Q}_{\phi\otimes \psi}^\Lambda (f \otimes g) (u, \omega) 	
			\\&& \bigg(\sqrt{\frac{b}{2\pi i}} \bigg)\int_{\mathbb{R}} g(x_1)    \overline{\psi}(x_1 - (u - m)) e^{i(ax_1^2 + b x_1 (\omega - \frac{2a}{b} m) + d x_1)} dx_1 \times e^{ia(- 2m^2 + 2um)} 
			e^{i(c \omega^2 + e\omega)} dm
			\\&=& - \bigg(\overline{\sqrt{\frac{b}{2\pi i}}}\bigg)   \int_{\mathbb{R}} \bigg(\sqrt{\frac{b}{2\pi i}} \bigg)  \int_{\mathbb{R}}  f(\tau) \overline{\phi}(\tau -m) e^{i(a\tau^2  +b \tau (\omega - \frac{2a}{b} u + \frac{2a}{b} m) + c (\omega - \frac{2a}{b} u + \frac{2a}{b} m)^2+ d \tau + e (\omega - \frac{2a}{b} u + \frac{2a}{b} m))} d\tau
			\\&& \bigg(\sqrt{\frac{b}{2\pi i}} \bigg)\int_{\mathbb{R}} g(x_1)    \overline{\psi}(x_1 - (u - m)) e^{i(ax_1^2 + b x_1 (\omega - \frac{2a}{b} m) + c (\omega - \frac{2a}{b} m)^2 + d x_1 + e(\omega - \frac{2a}{b} m))} dx_1 \\&& \times e^{ia(- 2m^2 + 2um)} 
			e^{ i(c \omega^2 + e\omega- c (\omega - \frac{2a}{b} u + \frac{2a}{b} m)^2- e (\omega - \frac{2a}{b} u + \frac{2a}{b} m) - c(\omega - \frac{2a}{b} m)^2 -e (\omega - \frac{2a}{b} m))} dm
			\\&=& - \bigg(\overline{\sqrt{\frac{b}{2\pi i}}}\bigg)   \int_{\mathbb{R}}  \int_{\mathbb{R}}  f(\tau) \overline{\phi}(\tau -m) \mathcal{K} (\tau, m_0)  d\tau
			\int_{\mathbb{R}} g(x_1)    \overline{\psi}(x_1 - (u - m)) \mathcal{K} (x_1, m_1) dx_1 \\&& \times e^{ia(- 2m^2 + 2um)} 
			e^{ i(c \omega^2 + e\omega- c (\omega - \frac{2a}{b} u + \frac{2a}{b} m)^2- e (\omega - \frac{2a}{b} u + \frac{2a}{b} m) - c(\omega - \frac{2a}{b} m)^2 -e (\omega - \frac{2a}{b} m))} dm,
		\end{eqnarray*}
		where $m_0 = (\omega - \frac{2a}{b} u + \frac{2a}{b} m)$ and $m_1 (\omega - \frac{2a}{b} m)$. Further simplifying, we get
		\begin{eqnarray*}
			&&\mathcal{Q}_{\phi\otimes \psi}^\Lambda (f \otimes g) (u, \omega) \\&=& - \bigg(\overline{\sqrt{\frac{b}{2\pi i}}}\bigg)   \int_{\mathbb{R}}  \int_{\mathbb{R}}  f(\tau) \overline{\phi}(\tau -m) \mathcal{K} (\tau, m_0)  d\tau
			\int_{\mathbb{R}} g(x_1)    \overline{\psi}(x_1 - (u - m)) \mathcal{K} (x_1, m_1) dx_1 \\&& \times e^{i(2aum + \frac{4a^2 c}{b^2} u m - \frac{8a^2c}{b^2} m^2 - 2am^2 - \frac{4a^2 c}{b^2} u^2 + \frac{4ac}{b} \omega u + \frac{2ae}{b} u - c \omega^2 - e \omega)} dm
			\\&=& - \bigg(\overline{\sqrt{\frac{b}{2\pi i}}}\bigg)  \bigg(\sqrt{\frac{2\pi i}{b}}\bigg)^2   \int_{\mathbb{R}}   \mathcal{Q}^\Lambda_\phi (f)(m_0, m) \mathcal{Q}^\Lambda_\psi (g)(m_1, u- m) \\&& \times e^{i(2aum + \frac{4a^2 c}{b^2} u m - \frac{8a^2c}{b^2} m^2 - 2am^2 - \frac{4a^2 c}{b^2} u^2 + \frac{4ac}{b} \omega u + \frac{2ae}{b} u - c \omega^2 - e \omega)} dm
			\\&=& - \bigg(\overline{\sqrt{\frac{b}{2\pi i}}}\bigg) \bigg(\sqrt{\frac{2\pi i}{b}}\bigg)^2 e^{i(-\frac{4a^2 c}{b^2} u^2 + \frac{4ac}{b} \omega u + \frac{2ae}{b} u - c \omega^2 - e \omega))}   \int_{\mathbb{R}}   \mathcal{Q}^\Lambda_\phi (f)(m_0, m) \mathcal{Q}^\Lambda_\psi (g)(m_1, u- m) \\&& \times e^{i(2aum + \frac{4a^2 c}{b^2} u m - \frac{8a^2c}{b^2} m^2 - 2am^2)} dm
			\\&=& \mathcal{E}(u, \omega) \int_{\mathbb{R}}   \mathcal{Q}^\Lambda_\phi (f)(m_0, m) \mathcal{Q}^\Lambda_\psi (g)(m_1, u- m) e^{i(2aum + \frac{4a^2 c}{b^2} u m - \frac{8a^2c}{b^2} m^2 - 2am^2)} dm,
		\end{eqnarray*}
		where $\mathcal{E}(u, \omega) = \bigg(\overline{\sqrt{\frac{b}{2\pi i}}}\bigg) \bigg(\sqrt{\frac{2\pi i}{b}}\bigg)^2 e^{i(-\frac{4a^2 c}{b^2} u^2 + \frac{4ac}{b} \omega u + \frac{2ae}{b} u - c \omega^2 - e \omega))}$.
	\end{proof}

	\section{Existence theorems of the convolution for the WQPFT}
	
	\begin{lemma}
		Let $\phi, \psi \in L^2(\mathbb{R}) \setminus \{0\}$ be window functions. Then for any two functions $f, g \in L^2 (\mathbb{R})$, we have
		\begin{eqnarray*}
			\langle \mathcal{Q}_{\phi}^\Lambda (f) (u, \omega), \mathcal{Q}_{\psi}^\Lambda (g) (u, \omega) \rangle = \langle f, g \rangle \langle \phi, \psi \rangle.
		\end{eqnarray*}
		If $f = g$ and $\phi = \psi$, then
		\begin{eqnarray*}
			|| \mathcal{Q}_{\phi}^\Lambda (f)||^2 = ||f||^2 ||\phi||^2.
		\end{eqnarray*}
	\end{lemma}

	\begin{proof}\cite{wlct1}
		We first assume that the windows $\phi, \psi$ are in $L^1(\mathbb{R}) \cap L^\infty (\mathbb{R}) \subset L^2(\mathbb{R})$. So that $f(x) \overline{\phi (x- u)}$ and $g(x) \overline{\psi (x- u)}$ are in $L^2(\mathbb{R})$ for all $u \in \mathbb{R}$. Therefore by Parseval's formula for QPFT \eqref{parseval} applies to $\omega-$integral and yields
		\begin{eqnarray*}
			\langle \mathcal{Q}_{\phi}^\Lambda (f) (u, \omega), \mathcal{Q}_{\psi}^\Lambda (g) (u, \omega) \rangle
			&=& \int_{\mathbb{R}} \int_{\mathbb{R}} \mathcal{Q}_{\phi}^\Lambda (f) (u, \omega) \overline{\mathcal{Q}_{\psi}^\Lambda (g) (u, \omega)} du d\omega
			\\&=& \int_{\mathbb{R}} \int_{\mathbb{R}} \mathcal{Q}_\Lambda (\overline{T_u \phi(x)}f(x)) (u, \omega) \overline{\mathcal{Q}_\Lambda (\overline{T_u \psi(x)}g(x)) (u, \omega)} du d\omega
			\\&=& \int_{\mathbb{R}} \Big( \int_{\mathbb{R}} \mathcal{Q}_\Lambda (\overline{T_u \phi(x)}f(x)) (u, \omega) \overline{\mathcal{Q}_\Lambda (\overline{T_u \psi(x)}g(x)) (u, \omega)} d\omega\Big) du 
			\\&=& \int_{\mathbb{R}} \Big( \int_{\mathbb{R}} \overline{T_u \phi(x)}f(x) \overline{\overline{T_u \psi(x)}g(x)} dx \Big) du
			\\&=& \int_{\mathbb{R}} \Big( \int_{\mathbb{R}} \overline{\phi(x-u)}f(x) \overline{g(x)}  \psi(x-u) dx \Big) du.
		\end{eqnarray*}
		Here $f(x)\overline{g(x)} \in L^1(\mathbb{R})$ and $\phi(x-u)\overline{\psi(x-u)} \in L^1(\mathbb{R})$, therefore Fubini's theorem allows us to interchange the order of integration. Hence
		\begin{eqnarray*}
			\langle \mathcal{Q}_{\phi}^\Lambda (f) (u, \omega), \mathcal{Q}_{\psi}^\Lambda (g) (u, \omega) \rangle 	  &=& \int_{\mathbb{R}} f(x) \overline{g(x)} \Big( \int_{\mathbb{R}} \overline{\phi(x-u)}\psi(x-u) du \Big) dx
			\\&=& \langle f, g \rangle \overline{\langle \phi, \psi \rangle}.
		\end{eqnarray*}
	\end{proof}
	\begin{remark}
		This theorem can be interpreted as preservation of enegy by WQPFT.
	\end{remark}
	As a particular case, if we take $f=g$ and $\phi = \psi$, we get the following corollary.
	\begin{corollary}
		If $f \in L^2(\mathbb{R})$ and $\phi \in L^2(\mathbb{R}) \setminus \{0\}$, then
		\begin{eqnarray*}
			||\mathcal{Q}_{\phi}^\Lambda (f)||_2 = ||f||_2 ||\phi||_2.
		\end{eqnarray*}
		If we take $||\phi||_2 =1$, then
		\begin{eqnarray*}
			||\mathcal{Q}_{\phi}^\Lambda (f)||_2 = ||f||_2, \text{ for all } f \in L^2(\mathbb{R}).
		\end{eqnarray*}
		This makes WQPFT an isometry from $L^2(\mathbb{R})$ to $L^2(\mathbb{R})$.
	\end{corollary}
	
	\begin{theorem}
		Let $f \in L^1(\mathbb{R}) \cap L^2(\mathbb{R})$ and window function $\phi \in L^p(\mathbb{R}) \setminus \{0\}$. Then
		\begin{eqnarray*}
			|| \mathcal{Q}_{\phi}^\Lambda (f)||_p \leq \sqrt{\frac{|b|}{2\pi}} ||f||_1 ||\phi||_p.
		\end{eqnarray*} 
	\end{theorem}
	\begin{proof} \cite{Prasad3}
		Using Minkowski's inequality, we obtain
		\begin{eqnarray*}
			|| \mathcal{Q}_{\phi}^\Lambda (f)||_p  &=& \bigg(\int_{\mathbb{R}}\Big| \int_{\mathbb{R}} f(x) \overline{\phi (x- u)} 	\mathcal{K}_{\Lambda} (\omega, x)dx\Big|^p du\bigg)^\frac{1}{p}
			\\&\leq& \int_{\mathbb{R}} \bigg(\int_{\mathbb{R}}\Big| f(x) \overline{\phi (x- u)} 	\mathcal{K}_{\Lambda} (\omega, x)\Big|^p du \bigg)^\frac{1}{p} dx
			\\&=& \sqrt{\frac{|b|}{2\pi}}\int_{\mathbb{R}} \bigg( \int_{\mathbb{R}} \Big| f(x) \overline{\phi (x- u)} \Big|^p du \bigg)^\frac{1}{p} dx
			\\&=& \sqrt{\frac{|b|}{2\pi}}\int_{\mathbb{R}} |f(x)| dx \Big( \int_{\mathbb{R}} |\overline{\phi (x- u)}|^p du \Big)^\frac{1}{p}
			\\&=& \sqrt{\frac{|b|}{2\pi}} ||f||_1 ||\phi||_p.
		\end{eqnarray*}
		Hence, we obtain the required result.
	\end{proof}

		\begin{theorem}
			Let $f, g \in L^p(\mathbb{R})$ and $\phi \in L^q(\mathbb{R}) \setminus \{0\}$ be a window function, where $\frac{1}{p} + \frac{1}{q} =1$, $1 \leq p, q < \infty$. Then 
			\begin{eqnarray*}
				||f \odot g||_{L^p(\mathbb{R})} \leq \sqrt{\frac{2 \pi}{|b|}} ||f||_{L^p(\mathbb{R})} ||g||_{L^p(\mathbb{R})} ||\phi||_{L^q(\mathbb{R})}.
			\end{eqnarray*}
		\end{theorem}
		\begin{proof}
			Replacing $f$ by the spectral convolution operator $f \odot g$ in the expression \eqref{p}, we have
			\begin{eqnarray*}
				||\mathcal{Q}_{\phi}^\Lambda (f \odot g)||_{L^p(\mathbb{R})} & = & \sqrt{\frac{|b|}{2\pi}} ||\phi||_{L^q(\mathbb{R})} ||f\odot g||_{L^p(\mathbb{R})}.  
			\end{eqnarray*}
			Taking p-norm on both sides of the formula of the spectral convolution theorem \eqref{wconv} yields
			\begin{eqnarray*}
				||\mathcal{Q}_\phi^\Lambda [f \odot g] (u,\omega)||_{L^p(\mathbb{R})} = ||\mathcal{Q}_\phi^\Lambda [f] (u,\omega) \mathcal{Q}_\phi^\Lambda [g] (u,\omega)||_{L^p(\mathbb{R})}. 
			\end{eqnarray*}
			Combining the above two equations gives
			\begin{eqnarray} \label{t1}
				\sqrt{\frac{|b|}{2\pi}} ||\phi||_{L^q(\mathbb{R})} ||f\odot g||_{L^p(\mathbb{R})} 
				& = & ||\mathcal{Q}_\phi^\Lambda [f] (u,\omega) \mathcal{Q}_\phi^\Lambda [g] (u,\omega)||_{L^p(\mathbb{R})} \nonumber
				\\&=& \bigg( \int_{\mathbb{R}} \int_{\mathbb{R}} |\mathcal{Q}_\phi^\Lambda [f] (u,\omega) \mathcal{Q}_\phi^\Lambda [g] (u,\omega)|^p du d\omega \bigg)^\frac{1}{p}.
			\end{eqnarray}
			Now substituting \eqref{p} in the RHS of the equation \eqref{t1}, we get
			\begin{eqnarray*}
				\sqrt{\frac{|b|}{2\pi}} ||\phi||_{L^q(\mathbb{R})} ||f\odot g||_{L^p(\mathbb{R})} &=& \bigg( \int_{\mathbb{R}} \int_{\mathbb{R}} |\mathcal{Q}_\phi^\Lambda [f] (u,\omega)|^p |\mathcal{Q}_\phi^\Lambda [g] (u,\omega)|^p du d\omega \bigg)^\frac{1}{p}
				\\&\leq& \bigg( \int_{\mathbb{R}} \int_{\mathbb{R}} \big(||f||_{L^p(\mathbb{R})} ||\phi||_{L^q(\mathbb(R))} \big)^p |\mathcal{Q}_\phi^\Lambda [g] (u,\omega)|^p du d\omega \bigg)^\frac{1}{p}
				\\&=& ||f||_{L^p(\mathbb{R})} ||\phi||_{L^q(\mathbb(R))} \bigg( \int_{\mathbb{R}} \int_{\mathbb{R}}  |\mathcal{Q}_\phi^\Lambda [g] (u,\omega)|^p du d\omega \bigg)^\frac{1}{p}
				\\&=& ||f||_{L^p(\mathbb{R})} ||\phi||_{L^q(\mathbb(R))}  ||\mathcal{Q}_\phi^\Lambda [g]||_{L^p(\mathbb{R})}
				\\&=& ||f||_{L^p(\mathbb{R})} ||g||_{L^p(\mathbb{R})} ||\phi||^2_{L^q(\mathbb{R})}.
			\end{eqnarray*}
			Hence the desired result is established.
		\end{proof}
		Here we have established a bound for the convolution operator of two functions. Now the following theorem \ref{bd} provides a bound for the WQPFT of the convolution of two functions.
	
			\begin{definition}(Young's inequality)
				Let $f \in L^p(\mathbb{R})$ and $g \in L^q(\mathbb{R})$ and $\frac{1}{p} + \frac{1}{q} = 1 + \frac{1}{r}$, then $f * g \in L^r(\mathbb{R})$. 
				\begin{eqnarray*}
					||f * g ||_r \leq ||f||_p ||g||_q,
				\end{eqnarray*}
				where $*$ is the FT convolution given in equation \eqref{clcon}.
			\end{definition}
			\begin{theorem}\label{bd}
				Let $f \in L^p(\mathbb{R})$, $g \in L^p(\mathbb{R})$, $\phi \in L^q(\mathbb{R})$ is a window function, $1 \leq p, q < \infty$, $\frac{1}{p} + \frac{1}{q} - 1 \geq 0$ and $f * g \in L^r(\mathbb{R})$, $\frac{1}{r} = \frac{1}{p} + \frac{1}{q} - 1$. Then
				\begin{eqnarray*}
					||\mathcal{Q}_\phi^\Lambda [f \odot g] (u,\omega)||_{L^r(\mathbb{R})}  \leq \sqrt{b}~~ ||f||_{L^p(\mathbb{R})} ||g||_{L^p(\mathbb{R})} ||\phi||^2_{L^q(\mathbb{R})}.
				\end{eqnarray*}
			\end{theorem}
			\begin{proof}
				The convolution operation given in \eqref{conv} can be rewritten as
				\begin{eqnarray*}
					(f \otimes g)(x) 
					&=& \sqrt{\frac{b}{i}} e^{-iax^2} \Big(\tilde{f}(x) * \tilde{g}(x) \Big),
				\end{eqnarray*}
				where $\tilde{f}(x) = f(x) e^{i a x^2}$.
				According to Young's inequality, we have the expression
				\begin{eqnarray*}
					||f \otimes g||_{L^r(\mathbb{R})} \leq \sqrt{b}~~ ||\tilde{f}||_{L^p(\mathbb{R})} ||\tilde{g}||_{L^q(\mathbb{R})}.
				\end{eqnarray*}
				Clearly, $||\tilde{f}||_{L^p(\mathbb{R})} = ||f||_{L^p(\mathbb{R})}$ and $||\tilde{g}||_{L^q(\mathbb{R})} = ||g||_{L^q (\mathbb{R})}$.
				So
				\begin{eqnarray*}
					||f \otimes g||_{L^r(\mathbb{R})} \leq \sqrt{b}~~ ||f||_{L^p(\mathbb{R})} ||g||_{L^q(\mathbb{R})}.
				\end{eqnarray*}
				By using the relation given in equation \eqref{rel}, we obtain
				\begin{eqnarray*}
					||\mathcal{Q}_\phi^\Lambda [f] (u,\omega)||_{L^r(\mathbb{R})} 	= ||f_\omega \otimes \phi_u (u)||_{L^r(\mathbb{R})} \leq \sqrt{b}~~ ||f||_{L^p(\mathbb{R})} ||\phi||_{L^q(\mathbb{R})},
				\end{eqnarray*}
				where $f_\omega(x) =  f(x) e^{i(b x \omega + c \omega^2 + e\omega)}$  and $\phi_u(u-x) =  \overline{\phi (-(u - x))} e^{i\{dx-ax^2+2aux\}}$.
				\\Replacing $f$ with $f \odot g$ in the above equation gives
				\begin{eqnarray*}
					||\mathcal{Q}_\phi^\Lambda [f \odot g] (u,\omega)||_{L^r(\mathbb{R})}  \leq \sqrt{b}~~ ||f \odot g||_{L^p(\mathbb{R})} ||\phi||_{L^q(\mathbb{R})}.
				\end{eqnarray*}
				By using the above theorem
				\begin{eqnarray*}
					||\mathcal{Q}_\phi^\Lambda [f \odot g] (u,\omega)||_{L^r(\mathbb{R})}  \leq \sqrt{b}~~ ||f||_{L^p(\mathbb{R})} ||g||_{L^q(\mathbb{R})} ||\phi||^2_{L^q(\mathbb{R})}.
				\end{eqnarray*}
				This completes the proof.
			\end{proof}
			\begin{corollary}
				Let $f \in L^p(\mathbb{R})$, $g \in L^1(\mathbb{R})$, $\phi \in L^q(\mathbb{R}) \setminus \{0\}$ is a window function, $1 \leq p, q < \infty$, $\frac{1}{p} + \frac{1}{q} = 1$ and $f * g \in L^p(\mathbb{R})$. Then
				\begin{eqnarray*}
					||\mathcal{Q}_\phi^\Lambda [f \odot g] (u,\omega)||_{L^p(\mathbb{R})}  \leq \sqrt{b}~~||f||_{L^p(\mathbb{R})} ||g||_{L^1(\mathbb{R})} ||\phi||^2_{L^q(\mathbb{R})}.
				\end{eqnarray*}
			\end{corollary}
			
			\section{Application}
			In this section, the solvability of a class of convolution equations associated with the WQPFT is established.
			
			Consider the convolution equation
			\begin{eqnarray} \label{appeq}
				\tau \nu(x) + (h \odot \nu)(x) = f(x),
			\end{eqnarray}
			where $\tau \in \mathbb{C}$, $f, h \in L^2(\mathbb{R})$ are given, $\nu$ is unknown and $\odot$ is as given in \eqref{wconv}.
			In order to obtain main result, we give the following lemma, which can be proved in a similar way  as that of  proposition 7 in \cite{castro1} or proposition 1 in \cite{castro2}.
			\begin{lemma}\label{lem2} \cite{castro1} \cite{castro2}
				Let $h \in L^2(\mathbb{R})$. Suppose $\Psi (u, \omega) = \tau  + \mathcal{Q}_\phi^\Lambda [h] (u,\omega)$,
				then 
				\begin{enumerate}
					\item If $\tau \neq 0$ and for any fixed $u \in \mathbb{R}$, then there exist a constant $\xi > 0$, such that $\Psi(u, \omega) \neq 0$ for every $|\omega| > \xi$.
					\item If $\Psi(u, \omega) \neq 0$ for any fixed $u \in \mathbb{R}$ and all $\omega \in \mathbb{R}$, then $\frac{1}{\Psi(u, \omega)}$ is bounded and continuous on $\mathbb{R}$.
				\end{enumerate}
			\end{lemma}
			
			\begin{proof}
				\begin{enumerate}
					\item By the Riemann- Lebesgue lemma for WQPFT \ref{rl}, the function $\Psi(u, \omega)$ is continuous on $\mathbb{R}$ and
					\begin{eqnarray*}
						\lim_{|\omega| \to \infty} \Psi(u, \omega) = \tau \neq 0,
					\end{eqnarray*}
					i.e., $\Psi(u, \omega)$ takes the value $\tau$ at infinity. Since $\tau \neq 0$ and $\Psi(u, \omega)$ is continuous, there exist a $\xi >0$ such that $\Psi(u, \omega) \neq 0$ for every $|\omega| \geq \xi$.
					\item Due to continuity of $\Psi$  and $\lim_{|\omega| \to \infty} \Psi(u, \omega) = \tau \neq 0$, there exist $R_0 > 0, \epsilon_1 > 0$ such that
					\begin{eqnarray*}
						\inf_{|\omega|> R_0} |\Psi(u, \omega)| > \epsilon_1. 
					\end{eqnarray*} 
					As $\Psi$ is continuous and not vanishing on the compact set $S(0, R_0) = \{ \omega \in \mathbb{R} : |\omega| \leq R_0 \}$, there exists $\epsilon_2 > 0$ such that $\inf_{|\omega| \leq R_0} |\Psi(u, \omega)| > \epsilon_2$. We deduce
					\begin{eqnarray*}
						\frac{1}{\Psi(u, \omega)} \leq \max \bigg\{ \frac{1}{\epsilon_1}, \frac{1}{\epsilon_2}\bigg\} < \infty.
					\end{eqnarray*}
					This implies that the function $\frac{1}{\Psi(u, \omega)}$ is continuous and bounded on $\mathbb{R}$. Hence the lemma is proved.
				\end{enumerate}
			\end{proof}

			\begin{theorem}
				\cite{gao} For any fixed $u \in \mathbb{R}$ and all $\omega \in \mathbb{R}$. Suppose that one of the following two conditions holds:
				\begin{enumerate}
					\item 
					\begin{eqnarray*}
						\tau \neq 0, \mathcal{Q}_\phi^\Lambda [g] (u,\omega) \neq 0 \text{ and } \mathcal{Q}_\phi^\Lambda [f] (u,\omega) \in L^2(\mathbb{R});
					\end{eqnarray*}
					\item 
					\begin{eqnarray*}
						\tau=0 \text{ and } \frac{\mathcal{Q}_\phi^\Lambda [f] (u,\omega)}{\mathcal{Q}_\phi^\Lambda [h] (u,\omega)} \in L^2(\mathbb{R}).
					\end{eqnarray*}
				\end{enumerate}
				Then convolution equation \eqref{appeq} has a solution in $L^2(\mathbb{R})$ if and only if $\mathcal{Q}_\phi^{-\Lambda} \Big(\frac{\mathcal{Q}_\phi^\Lambda [f]}{\Psi}\Big) \in L^2(\mathbb{R})$. Furthermore, the solution has the form of $\nu = \mathcal{Q}_\phi^{-\Lambda} \Big(\frac{\mathcal{Q}_\phi^\Lambda [f]}{\Psi}\Big) \in L^2(\mathbb{R}).$
			\end{theorem}
			
			\begin{proof}
				Necessary part:
				\\Suppose that the convolution equation \eqref{appeq} has a solution $\nu \in L^2(\mathbb{R})$. Applying WQPFT on both sides of the equation \eqref{appeq} becomes
				\begin{eqnarray*}
					\mathcal{Q}_\phi^\Lambda \big[\tau \nu(x) + (h \odot \nu)(x)\big] (u, \omega)  = \mathcal{Q}_\phi^\Lambda [f(x)] (u, \omega)
					\\\mathcal{Q}_\phi^\Lambda \big[\tau  \nu(x) \big] (u, \omega) + \mathcal{Q}_\phi^\Lambda \big[(h \odot \nu)(x)\big] (u, \omega) = \mathcal{Q}_\phi^\Lambda [f(x)] (u, \omega).
				\end{eqnarray*} 
				According to the formula \eqref{wconv}, we get
				\begin{eqnarray*}
					\tau \mathcal{Q}_\phi^\Lambda[\nu] (u, \omega) + \mathcal{Q}_\phi^\Lambda [h](u, \omega) \mathcal{Q}_\phi^\Lambda[\nu] (u, \omega) = \mathcal{Q}_\phi^\Lambda [f] (u, \omega),
				\end{eqnarray*}
				which implies
				\begin{eqnarray*}
					\big( \tau + \mathcal{Q}_\phi^\Lambda [h](u, \omega) \big) \mathcal{Q}_\phi^\Lambda[\nu] (u, \omega) = \mathcal{Q}_\phi^\Lambda [f] (u, \omega).
				\end{eqnarray*}
				If $\tau \neq 0$ and $\mathcal{Q}_\phi^\Lambda [g](u, \omega) \neq 0$, then $\Psi (u, \omega) = \tau + \mathcal{Q}_\phi^\Lambda [h] (u,\omega) \neq 0$. The above equation becomes
				\begin{eqnarray}\label{eq}
					\mathcal{Q}_\phi^{\Lambda} [\nu] (u, \omega) = \frac{\mathcal{Q}_\phi^\Lambda [f](u, \omega)}{\Psi (u, \omega)}.
				\end{eqnarray}
				Based on Lemma \ref{lem2}, $\frac{1}{\Psi(u, \omega)}$ is bounded and continuous on $\mathbb{R}$ and $\mathcal{Q}_\phi^{\Lambda}[f](u, \omega) \in L^2(\mathbb{R})$, we have $\mathcal{Q}_\phi^{\Lambda}[\nu](u, \omega) \in L^2(\mathbb{R})$. Applying the inverse transform of the WQPFT to the above equation \eqref{eq}, we obtain the result.
				\\Sufficient part:
				\\Let
				\begin{eqnarray*}
					\nu = \mathcal{Q}_\phi^{-\Lambda} \Big(\frac{\mathcal{Q}_\phi^\Lambda [f]}{\Psi}\Big),
				\end{eqnarray*}
				then, we have, $\nu \in L^2(\mathbb{R})$. So, we obtain $\mathcal{Q}_\phi^{\Lambda}[\nu] =  \frac{\mathcal{Q}_\phi^\Lambda [f]}{\Psi}$. 
				\\i.e.,
				\begin{eqnarray*}
					\Psi \mathcal{Q}_\phi^{\Lambda}[\nu] = \mathcal{Q}_\phi^\Lambda [f].
				\end{eqnarray*}
				Then
				\begin{eqnarray*}
					\big( \tau + \mathcal{Q}_\phi^\Lambda [h](u, \omega) \big) \mathcal{Q}_\phi^\Lambda[\nu] (u, \omega) = \mathcal{Q}_\phi^\Lambda [f] (u, \omega).
				\end{eqnarray*}
				Based on the uniqueness of $\mathcal{Q}_\phi^\Lambda$, then $\nu$ satisfies the convolution equation \eqref{appeq}, the case (1) is proved.
				\\The case of (2) may be proved similarly to that of case (1).  
			\end{proof}

			\section{Conclusion}
			In this article, we have developed a theoretical foundation for the WQPFT, a time-frequency representation that effectively combines windowing with quadratic phase modulation. We investigated several fundamental properties and structural results of the WQPFT, including the reproducing kernel, a reconstruction formula, and a characterization of the range of the transform. Convolution theorems were established in both the spectral and spatial domains, and we proved the existence of solutions and derived norm estimates for the corresponding convolution operations. Moreover, we demonstrated how the convolution theorem naturally provides a solution to certain convolution-type integral equations within the framework of WQPFT. Building on these foundations, there are many possible directions for future research in WQPFT. Future work may focus on developing numerical algorithms for efficient computation of the WQPFT and applying it to practical problems such as filtering, denoising, and feature extraction. Extensions to higher dimensions, quaternion domains, and discrete settings, as well as the study of uncertainty principles and sampling theorems, also present promising research directions.
			
			\section*{Funding statement}
			This research received no specific grant from any funding agency in the public, commercial, or not-for-profit sectors.
			
			\section*{Declaration of competing interest}
			The authors declare that they have no known competing financial interests or personal relationships that could have appeared to influence the work reported in this paper.
			\\\\
			\textbf{Note}: The manuscript has no associated data.
			
			\bibliographystyle{amsplain}
			
		\end{document}